\documentclass{article}
\usepackage[utf8]{inputenc}
\usepackage[margin=1in]{geometry}
\usepackage{mathtools}
\usepackage{amsfonts}
\usepackage{amsmath}
\usepackage{amssymb}
\usepackage{amsthm}
\usepackage{enumitem}
\usepackage{thmtools}

\usepackage{hyperref}
\usepackage[english]{babel}
\usepackage[style=numeric, maxbibnames=99]{biblatex}
\usepackage{csquotes}

\hypersetup{
pdfencoding=auto,
psdextra,
breaklinks=true}

\declaretheorem[numberwithin=section]{theorem}
\declaretheorem{claim}
\declaretheorem[numberlike=theorem]{lemma,proposition,problem,corollary,conjecture}
\newcommand{\qqed}{{\renewcommand\qedsymbol{$\vartriangleleft$}\qed}}
\AtBeginEnvironment{proof}{\setcounter{claim}{0}}

\newcommand{\Fc}{\mathcal{F}}
\newcommand{\Gc}{\mathcal{G}}
\newcommand{\Lc}{\mathcal{L}}
\newcommand{\Tc}{\mathcal{T}}
\newcommand{\Wc}{\mathcal{W}}
\newcommand{\hash}{\mathbin\#}
\DeclarePairedDelimiter{\abs}{\lvert}{\rvert}
\DeclareMathOperator{\Forb}{Forb}

\title{On the joint embedding property for cographs and trees}
\author{Daniel Carter}
\date{}

\begin{document}

\maketitle

\begin{abstract}
A family of graphs $\Fc$ is said to have the \textit{joint embedding property} (JEP) if for every $G_1,G_2\in \Fc$, there is an $H\in \Fc$ that contains both $G_1$ and $G_2$ as induced subgraphs. If $\Fc$ is given by a finite set $S$ of forbidden induced subgraphs, it is known that determining if $\Fc$ has JEP is undecidable. We prove that this problem is decidable if $P_4\in S$ and generalize this result to families of rooted labeled trees under topological containment, bounded treewidth families under the graph minor relation, and bounded cliquewidth families under the induced subgraph relation.
\end{abstract}

\section{Introduction}

Graphs (including trees) in this paper are finite, and undirected graphs are simple. Suppose $(\Fc,\le)$ is a partially ordered set. Then $\Fc$ is said to have the \textit{joint embedding property} (JEP) if for every $X_1,X_2\in \Fc$, there is a $Y\in \Fc$ with $X_1\le Y$ and $X_2\le Y$. We say $X$ is \textit{$Y$-free} if $Y\not\le X$, and for a set $S\subseteq \Fc$, $X$ is said to be \textit{$S$-free} if $X$ is $Y$-free for all $Y\in S$. We write $\Forb_\le(S)$ to denote the set of $S$-free $X$. Note that $\Forb_\le(S)$ is still a partially ordered set under $\le$.

We consider the following computational problem:
\begin{problem}\label{prob:jep}
Let $(\Gc,\subseteq)$ be the family of graphs partially ordered by the induced subgraph relation. Given a finite set of graphs $S$, does $\Forb_\subseteq(S)$ have JEP?
\end{problem}
We say that $X_1,X_2\in \Fc$ form an \textit{$\Fc$-bad pair} (or just ``bad pair'' if $\Fc$ is understood) if there is no $Y\in \Fc$ that contains both $X_1$ and $X_2$. Obviously $\Fc$ has JEP if and only if there are no bad pairs.

In~\cite{undecidable}, it was shown that Problem~\ref{prob:jep} is undecidable. In contrast, we prove that this problem is decidable if $P_4\in S$, where $P_4$ is the path graph with four vertices. $P_4$-free graphs are also known as \textit{cographs}. The question of the decidability of JEP for subclasses of cographs was raised by Lozin \cite{question}. The crucial fact used in the proof is that $P_4$-free graphs are determined by their ``cotree'' structure, also known as modular decomposition. In fact, we actually prove a much more general result about tree languages, which we now describe.

From now on, unless specified otherwise, when we write ``tree'' we mean ``finite, rooted, labeled, binary tree''. Say that a tree $T$ contains $U$ \textit{topologically} if $U$, including the choice of root, may be obtained from a sequence of edge contractions and vertex/edge deletions in $T$, where the label of the new vertex representing a contracted edge is the label of the parent (this is the usual notion of topological containment of rooted labeled trees). Alternatively, we may ask for an injection $\phi:V(U)\hookrightarrow V(T)$ such that the paths in $T$ from $\phi(u)$ to $\phi(v)$ are edge-disjoint for every edge $uv\in E(U)$, contracting all such paths to edges and deleting all vertices not on any of these paths results in $U$, and $v$ is a descendent of $u$ in $U$ if and only if $\phi(v)$ is a descendent of $\phi(u)$ in $T$.

\begin{problem}\label{prob:treejep}
Let $(\Tc_k,\preceq)$ be the set of trees partially ordered by topological containment, where there are $k$ possible labels. Given a finite set of trees $S$, does $\Forb_\preceq(S)$ have JEP?
\end{problem}
Our main theorem, proved in Section~\ref{sec:main}, is the following:
\begin{theorem}\label{thm:trees}
Problem~\ref{prob:treejep} is decidable. In fact, it is decidable to determine if $(L,\preceq)$ has JEP for any regular language of trees $L$, if a tree automaton recognizing $L$ is given as input.
\end{theorem}
This also includes, as an easy consequence, deciding JEP for regular languages of strings under the relation $a\preceq b$ if $a$ is a (not necessarily contiguous) substring of $b$, generalizing the main result of \cite{subword}, which proved this result for string languages defined by (finitely many) forbidden strings, a special case of regular languages of strings. We prove this result separately, also in Section~\ref{sec:main}, as a stepping stone to the full result.

Our main theorem has a number of corollaries, proved in Section~\ref{sec:cor}. First, the restriction to binary trees is not necessary:
\begin{corollary}\label{cor:generaltrees}
Let $(\Tc_k',\preceq)$ be the set of rooted, labeled, not necessarily binary trees partially ordered by topological containment, where there are $k$ possible labels. It is decidable to determine if $\Forb_\preceq(S)$ has JEP for any finite set $S\subseteq\Tc_k'$.
\end{corollary}

Additionally, the main theorem may be applied to decide if a given subclass of cographs has JEP under the induced subgraph relation:
\begin{corollary}\label{cor:p4free}
There is an algorithm to decide Problem~\ref{prob:jep} for $S$ with $P_4\in S$.
\end{corollary}

We also show that it is decidable to determine if bounded treewidth families have JEP under the graph minor relation:
\begin{corollary}\label{cor:tw}
If $(\Gc, \preceq)$ is the set of graphs under the minor relation, and $S$ is a (finite) set of graphs such that every graph in $\Forb_\preceq(S)$ has treewidth at most some constant $k$, then it is decidable to determine if $\Forb_\preceq(S)$ has JEP.
\end{corollary}
The condition that $\Forb_\preceq(S)$ has bounded treewidth is equivalent to $S$ containing a planar graph \cite{treewidth}. It is also decidable to determine if bounded treewidth families have JEP under the induced subgraph relation, and even more strongly:
\begin{corollary}\label{cor:cw}
If $S$ is a given finite set of graphs and $k$ is a given integer, then it is decidable to determine if the set of $S$-free graphs with cliquewidth at most $k$ has JEP under the induced subgraph relation.
\end{corollary}
Note that $P_4$-free graphs are precisely the graphs with cliquewidth at most 2 \cite{cw}, so this also generalizes Corollary~\ref{cor:p4free}.

The main theorem is proved by exhibiting upper bounds on the size of ``minimal'' bad pairs, then using this information to construct a tree automaton that, roughly speaking, accepts a pair $(T_1, T_2)$ if and only if it is a bad pair. Standard techniques can then be used to determine if the set of bad pairs is empty or not. Moreover, we immediately obtain a linear time algorithm to check if a given pair is bad or not, though the hidden constants make this algorithm highly impractical.

The existence of this tree automaton shows that the set of bad pairs of a regular language of trees under topological containment itself forms a regular tree language in some sense. This fact can actually easily be proven by using the fact that trees are well-quasi-ordered (wqo) under topological containment, though this does not give an explicit construction of an automaton recognizing the desired language, so it does not actually prove Theorem~\ref{thm:trees}. Nevertheless, this suggests that there may be a deeper relationship between wqo families and the decidability of the joint-embedding property. We very tentatively conjecture the following:
\begin{conjecture}\label{conj:wqo}
Suppose that $(\Fc, \le)$ is wqo and there are algorithms to decide some basic problems related to $\le$, e.g.:
\begin{enumerate}
    \item ``Is $X\le Y$?''
    \item ``Given $X$, find the set of $Z$ such that $X\le Z$ and there is no $Y$ with $X\le Y\le Z$ except $X$ and $Z$.''\footnote{Note that the set of such $Z$ forms an antichain under $\le$, and is therefore finite by wqo.}
\end{enumerate}
Then the following problem is decidable: ``Given a finite set $S\subseteq \Fc$, does $\Forb_\le(S)$ have JEP?''
\end{conjecture}
There may be some other subproblems that require algorithms besides the two listed in order for this problem to be decidable. Theorem~\ref{thm:trees} shows that if $(\Fc,\le)$ can be computably embedded in $(\Tc_k, \preceq)$ for some $k$, i.e.\ if there is a computable function $\phi:\Fc\to \Tc_k$ such that $x\le y$ if and only if $\phi(x)\preceq \phi(y)$, then determining if $\Forb_\le(S)$ has JEP is decidable. Thus, most ``sufficiently simple'' wqo families satisfy this conjecture. A possible natural counterexample to any form of this conjecture is the set of graphs under the minor relation, which cannot be embedded into $\Tc_k$ for any $k$ (this is implied by results in \cite{meta}). We discuss this conjecture, and others, in Section~\ref{sec:more}.

\section{JEP for tree languages}
\label{sec:main}

In this section we prove Theorem~\ref{thm:trees}. We make heavy use of automata theory; in particular, the notions of (deterministic, bottom-up) \textit{tree automata} and \textit{regular tree languages}, also known as recognizable tree languages. We recommend referring to \cite{tata} for readers unfamiliar with these notions.

\subsection{Explicitly and implicitly regular families}

We distinguish between families of languages that are ``explicitly'' or ``implicitly'' regular: a family $\Lc$ of languages (of trees or strings) is said to be \textit{explicitly regular} if there is an algorithm that, given (a description of) $L\in\Lc$, constructs a finite automaton $M$ that accepts precisely $L$; and $\Fc$ is \textit{implicitly regular} if every $L\in\Lc$ is known to be regular but there is not necessarily an algorithm that produces a finite automaton recognizing $L$. These notions are actually different; for example:
\begin{proposition}
Let $T$ be a Turing machine and $L_T$ the language $\{1^n\mid T\text{ does not halt in time $\le n$}\}$. Then $\{L_T\}_T$ is an implicitly regular family that is not explicitly regular.
\end{proposition}
\begin{proof}
For any Turing machine $T$, $L_T$ either has finitely many strings in the case $T$ halts, or all strings of the form $1^n$ in the case $T$ does not halt. In either case, $L_T$ is regular, so $\{L_T\}_T$ is an implicitly regular family.

If there was an algorithm that produced, for any Turing machine $T$, a finite automaton recognizing $L_T$, then the halting problem would be decidable. Indeed, it is easy to determine if a given finite automaton rejects some string of the form $1^n$, and $T$ halts if and only if $L_T$ rejects some such string.
\end{proof}
Note that determining if $x\in L_T$ actually \textit{is} decidable for any string $x$: reject if $x$ is not of the form $1^n$; otherwise, simulate $T$ for $n$ steps to determine if $T$ halts in that time or not. We may also say, slightly informally, that a particular language $L$ is ``explicitly regular'' in the case there is an explicit construction of a finite automaton $M$ recognizing $L$.

We make use of the following standard or obvious results:
\begin{proposition}\label{prop:basic}
The following facts hold:
\begin{enumerate}
    \item The set of $S$-free trees is a regular tree language. In fact, $\{\Forb_\preceq(S)\}_{\text{finite }S\subseteq \Tc_k}$ is an explicitly regular family of tree languages.
    \item Complements, unions, and intersections of finitely many (explicitly) regular tree languages are (explicitly) regular tree languages.
    \item Consider regular tree languages over the alphabet $\Sigma$. Define $X\hash Y$, where $X$ and $Y$ are trees, to be the tree over the alphabet $\Sigma\sqcup\{\hash \}$ with root labeled $\hash $ that has child subtrees $X$ and $Y$. If $L_1,L_2$ are (explicitly) regular tree languages over $\Sigma$, then $L_1\hash L_2\coloneq\{X\hash Y\mid X\in L_1,Y\in L_2\}$ is (explicitly) regular over $\Sigma\sqcup\{\hash \}$.
\end{enumerate}
\end{proposition}

The key lemma, which is in fact stronger than Theorem~\ref{thm:trees}, is the following, which we prove in the next subsection:
\begin{lemma}\label{lem:main}
Suppose $\Lc$ is an explicitly regular family of tree languages. Define
\[ B_L\coloneq\{X\hash Y\mid \text{$(X,Y)$ is an $L$-bad pair}\}. \]
Then $\{B_L\}_{L\in\Lc}$ is an explicitly regular family.
\end{lemma}
This immediately implies Theorem~\ref{thm:trees}, since it is decidable to determine if a given explicitly regular family is empty (and $B_L$ is empty if and only if $L$ has JEP).

In fact, we will prove something even slightly stronger. Say that $(X,Y)$ is an \textit{$L$-semibad pair} if there is no $Z\in L$ with $X,Y\preceq Z$ (the difference between semibad and bad pairs is that $X$ and $Y$ need not be in $L$ for a semibad pair).
\begin{lemma}\label{lem:main2}
Suppose $\Lc$ is an explicitly regular family of tree languages. Define
\[ B_L'\coloneq\{X\hash Y\mid \text{$(X,Y)$ is an $L$-semibad pair}\}. \]
Then $\{B_L'\}_{L\in\Lc}$ is an explicitly regular family.
\end{lemma}
\begin{proof}[Proof that Lemma~\ref{lem:main2} implies Lemma~\ref{lem:main}]
By Proposition~\ref{prop:basic}, the language $B_L'\cap(L\hash L)$ is explicitly regular assuming $B_L'$ and $L$ are explicitly regular. This language is precisely $B_L$.
\end{proof}
We note that we may prove that these families are \textit{implicitly} regular very easily using the fact that $(\Tc_k,\preceq)$ is wqo:
\begin{proposition}\label{prop:implicit}
For any $L$, $B_L'$ is regular. In other words, $\{B_L'\}_{L\in\Lc}$ is an implicitly regular family where $\Lc$ is the set of all regular languages of trees in $\Tc_k$.
\end{proposition}
\begin{proof}
Since $\Tc_k$ is wqo under $\preceq$ \cite{kruskal}, the set $L^2$ is wqo under $\preceq_2$ with $(X_1,Y_1)\preceq_2(X_2,Y_2)$ if $X_1\preceq Y_1$ and $X_2\preceq Y_2$. Note that if $(X_1,Y_1)$ is an $L$-semibad pair and $(X_1,Y_1)\preceq_2(X_2,Y_2)$, then $(X_2,Y_2)$ is also semibad. Therefore there are finitely many \textit{minimal} (under $\preceq_2$) semibad pairs $(X,Y)$; let the set of such minimal semibad pairs be $R_L$.
Then by Proposition~\ref{prop:basic}, the following languages are regular:
\begin{enumerate}
    \item The set $L_{X,Y}$ of trees of the form $X'\hash Y'$ where $(X,Y)\preceq_2(X',Y')$ and $X',Y'\in L$, given $X,Y\in L$.
    \item $\bigcup_{(X,Y)\in R_L}L_{X,Y}$, using here the fact that $R_L$ is finite.
\end{enumerate}
The latter language is precisely $B_L'$.
\end{proof}
This argument seems like it is not all that useful for the purposes of proving Theorem~\ref{thm:trees}. Indeed, the proof uses essentially no information about $L$ at all! We discuss wqo further in Section~\ref{sec:more}.

\subsection{A more explicit construction}

In this subsection we prove Lemma~\ref{lem:main2}. First, we clarify the definition of a (deterministic) tree automaton $M$. $M$ contains the following data:
\begin{itemize}
    \item a finite \textit{label set} $\Lambda$,
    \item a finite \textit{state set} $\Sigma$,
    \item a subset $\Sigma_a\subseteq \Sigma$ of \textit{accepting states}.
    \item two \textit{transition functions} $m_0:\Lambda\to\Sigma$ and $m_2:\Lambda\times\Sigma^2\to\Sigma$. We may assume $m_2(x,y,z)=m_2(x,z,y)$ for all $x,y,z$, so the automaton has the same behavior on isomorphic trees.\footnote{This assumption is not strictly necessary, but it saves us having to distinguish between ``left'' and ``right'' children. The analogous result without this assumption can be obtained by replacing $\Lambda$ with $\{L, R\}\times\Lambda$ and enforcing that each non-leaf has one child with label of the form $(L,\lambda)$ (the ``left'' child) and one child with label of the form $(R, \mu)$ (the ``right'' child); this modified language is still (explicitly) regular.}
\end{itemize}
$M$ works by defining, for each vertex $u$ of $T$ starting from the leaves,
\[ \sigma_M(u)\coloneq\begin{cases}
    m_0(\lambda_u) & \text{if $u$ is a leaf} \\
    m_2(\lambda_u, \sigma_M(v_1), \sigma_M(v_2)) & \text{if $u$ has children $v_1$ and $v_2$,}
\end{cases} \]
where $\lambda_u$ is the label of vertex $u$; here $T$ is assumed to be labeled with labels from $\Lambda$. Then $M$ \textit{accepts} $T$ if $\sigma_M(r)\in \Sigma_a$, where $r$ is the root of $T$. We say $M$ \textit{recognizes} $L$ if $L$ is the set of trees that $M$ accepts. We write $M(T)$ to denote $\sigma_M(r)$ when $M$ is run on tree $T$ with root $r$. Our definitions differ only superficially from the definitions in \cite{tata} (when appropriately restricted to binary trees).

Before we prove Lemma~\ref{lem:main}, we will first prove the analogous result for regular languages of strings, which demonstrates most of the main ideas. The general outline for both proofs is as follows:
\begin{enumerate}[start=0]
    \item Build a special directed graph using an automaton recognizing $L$, and associate to each string/tree $z$ a ``walk'' $W_z$, of which there are finitely many possibilities.
    \item Show that the set of walks $\Wc_x=\{W_z\mid z\preceq x\text{ and }z\in L\}$ gives enough information to determine if a pair is semibad or not (i.e.\ whether or not $(x,y)$ is semibad can be determined just from $\Wc_x$ and $\Wc_y$).
    \item Show that the set of possible $\Wc_x$ sets is computable.
    \item Use this to find the set of minimal semibad pairs (as in the proof of Proposition~\ref{prop:implicit}), then construct an automaton to recognize $B_L'$.
\end{enumerate}

\begin{lemma}\label{lem:strings}
Let $(L, \preceq)$ be a partially ordered set of strings, where $L$ is an explicitly regular language and $a\preceq b$ if $a$ is a (not necessarily contiguous) substring of $b$. Then $B_L'=\{x\hash y\mid \text{$(x,y)$ is an $L$-semibad pair}\}$ is explicitly regular.
\end{lemma}
Here, in ``$x\hash y$'', $\hash$ is a special symbol that does not appear in the alphabet underlying $L$. This implies that determining if $(L,\preceq)$ has JEP is decidable if a deterministic finite automaton recognizing $L$ is given. This generalizes the main result of \cite{subword}, which is that it is decidable to determine if $\Forb_\preceq(S)$ has JEP for any finite set of strings $S$ where $\preceq$ is the non-necessarily-contiguous substring relation. However, their algorithm is extraordinarily more efficient than ours, as we discuss later.

For this proof and the next, we will need the concept of a line graph of a directed graph. If $G$ is a directed graph, the \textit{line graph} of $G$, $L(G)$, has vertices $E(G)$ and a directed edge from $uv$ to $vw$ for each $uv,vw\in E(G)$ (here $uv$ and $vw$ are also directed edges, from $u$ to $v$ and from $v$ to $w$, respectively). A loop in $G$ at vertex $v$ should be read as ``$uv$ with $u=v$'', so the vertex in $L(G)$ corresponding to this loop also has a loop in $L(G)$ (i.e.\ $uv$ still has an edge to $vw$ in the case $u=v=w$). A useful property is that if $C$ is a strongly connected component (scc) of $G$, then the edges between vertices in $C$ forms an scc in $L(G)$. If an edge $uv\in E(G)$ is labeled, the corresponding vertex in $L(G)$ gets the same label.
\begin{proof}[Proof of Lemma~\ref{lem:strings}]
Let $M$ be a deterministic finite (string) automaton recognizing $L$. We may view $M$ as an edge-labeled directed graph in the obvious way, where the vertices are states of $M$ and edges are transitions labeled by the appropriate element of $\Lambda$, the underlying alphabet. Let $C_1,\dots,C_k$ be the scc's of $L(M)$. The quotient $L(M)/\{C_1,\dots,C_k\}$ is a directed acyclic graph $D$ after loops are removed. Call a vertex $C_i$ of $D$ \textit{loopy} if it has a loop in the quotient $L(M)/\{C_1,\dots,C_k\}$ and \textit{non-loopy} otherwise; loopy vertices correspond to scc's of $L(M)$ that either have a loop or have at least two vertices. Call $C_i$ \textit{initial} if some out-edge of the start state of $M$ appears in $C_i$ and \textit{accepting} if some in-edge of an accepting state of $M$ appears in $C_i$.


Let $\Wc$ be the set of directed walks in $D$ starting from an initial vertex. Obviously $\Wc$ is finite and computable. For any string $z$ let $W_z$ be the walk in $D$ corresponding to the sequence of scc's that $M$ hits when run on input $z$. Note that every element of $\Wc$ is equal to $W_z$ for some string $z$. Let $\Wc_x=\{W_z\mid x\preceq z\text{ and }z\in L\}$.

\begin{claim}\label{claim:one}
For any $x,y$, there is a $z\in L$ with $x,y\preceq z$ if and only if $\Wc_x\cap \Wc_y\ne\varnothing$.
\end{claim}
    If $x,y\preceq z$ then $W_z\in \Wc_x\cap \Wc_y$ by definition. On the other hand, if $\Wc_x\cap \Wc_y$ is nonempty, say containing walk $W=C_{i_1}\dots C_{i_\kappa}$, then let $z$ be the following string. If $C_{i_\alpha}$ is loopy, let $c_\alpha$ be a walk in $C_{i_\alpha}\subseteq L(M)$ that contains every vertex of $C_{i_\alpha}$ at least $n$ times, where $n=\max\{\abs{x},\abs{y}\}$. If $C_{i_\alpha}$ is not loopy, then let $c_\alpha$ be the one-vertex walk consisting of the single vertex inside $C_{i_\alpha}$. This walk should be chosen to start at an out-neighbor of $C_{i_{\alpha-1}}$ (if $\alpha>1$) and end at an in-neighbor of $C_{i_{\alpha+1}}$ (if $\alpha<\kappa$), so that the concatenation $c_1\cdots c_\kappa$ represents a walk in $L(M)$. Additionally, $c_1$ should start at an out-edge (in $M$) of the start state of $M$ and $c_\kappa$ should end at an in-edge of an accepting state of $M$. The latter restriction is possible to fulfill by the assumption that $W\in \Wc_x$, so $C_{i_\kappa}$ is accepting.

Then $z$ is formed by reading off the labels of the vertices hit by $c_1\cdots c_\kappa$, so that $M$ run on $z$ uses precisely the transitions dictated by $c_1\cdots c_\kappa$. It is easy to see by induction on $n$ and $\kappa$ that $w\preceq z$ for \textit{every} string $w$ with $\abs{w}\le n$ and $W\in \Wc_w$; in particular, $z$ contains $x$ and $y$.
\qqed

Therefore for the purposes of finding minimal $L$-semibad pairs, it suffices to consider only the smallest possible string $x$ with a particular $\Wc_x$. Not all subsets of $\Wc$ are necessarily realized as $\Wc_x$ for some $x$, but any $\Wc_x$ is realized by a bounded-size string:
\begin{claim}
Consider $\Wc_x$ for arbitrary $x$. There is an $x^*$ with $\abs{x^*}\le 2^{\abs{\Wc}}$ and $\Wc_x=\Wc_{x^*}$.
\end{claim}
Suppose that $\abs{x}>2^{\abs{\Wc}}$. We will find a smaller $x'$ with $\Wc_{x'}=\Wc_x$. Define a deterministic automaton $N$ as follows:
\begin{itemize}
    \item The alphabet of $N$ is the same as $M$; call it $\Lambda$.
    \item The states of $N$ are $2^\Wc$.
    \item The start state of $N$ is the set containing all one-vertex walks $C_i$ where $C_i$ is initial.
    \item For each $S\subseteq\Wc$ and $\lambda\in\Lambda$, there is a transition from $S$ to $S'$ labeled $\lambda$ where
    \begin{align*}
        S'=\{W\mid{}&\text{there is a $W'\in S$ such that $W$ has prefix $W'$ and} \\
        &\text{the last vertex $C_i$ of $W$ contains some vertex labeled $\lambda$}\}.
    \end{align*}
\end{itemize}
The accepting states of $N$ are irrelevant for our discussion. Then it is not difficult to see by induction on $\abs{x}$ that $\Wc_x$ is precisely the set of walks $W$ in $\Wc$ such that a prefix of $W$ is in $N(x)$ (the state $N$ is left in after reading $x$) and the last scc $C_{i_\kappa}$ in $W$ is accepting.

We have that $N$ has $2^{\abs{\Wc}}$ states, so by our assumption on the length of $x$, there are $\beta_1<\beta_2$ such that the $\beta_1$st and $\beta_2$nd states hit by $N$ running on input $x$ are the same. Define $x'$ to be the string obtained from $x$ by deleting symbols $\beta_1+1$ through $\beta_2$. We have $N(x')=N(x)$ so $\Wc_{x'}=\Wc_x$ by the previous discussion, proving the claim.
\qqed

Therefore for any minimal $L$-semibad pair $(x,y)$, both $x$ and $y$ have length at most $2^{\abs{\Wc}}$. The above proof also gives a method to compute $\Wc_x$ for arbitrary $x$: run the automaton $N$ on input $x$ and check which walks $W\in N(x)$ are possible prefixes of walks that end in an accepting $C_{i_\kappa}$. This allows us to compute the set of minimal $L$-semibad pairs $R_L$, which is sufficient to construct an automaton to recognize $B_L'$. Indeed, $B_L'$ is simply
\[ (\Lambda^*\hash \Lambda^*) \cap \bigcap_{(x,y)\in R_L}\overline{L_x\hash L_y} \]
where $L$ has alphabet $\Lambda$ and $L_w$ is the set of strings $z$ with $w\preceq z$, which is explicitly regular (here $\Lambda^*$ is the Kleene star of $\Lambda$, the set of finite strings with alphabet $\Lambda$).
\end{proof}
Note that $\abs{\Wc}$ is not more than exponentially large in the number of transitions of $M$, and if $L=\Forb_\preceq(S)$ for some $S$, then the number of states in the smallest $M$ recognizing $L$ is not more than exponential in $n=\sum_{f\in S}\abs{f}$, so the upper bound we obtain on the size of the smallest $L$-semibad pair is triply exponential in $n$. Then there are at most quadruply exponentially many minimal $L$-semibad pairs, so smallest automaton recognizing $B_L'$ (and $B_L$) is not more than quintuply exponential in $n$. The algorithm in \cite{subword}, on the other hand, runs in time $O(\abs{S}n^2)$.

Now we proceed with the proof for tree languages. This proof follows the same general structure as the proof for string languages, but is rather more complicated, particularly the proof of the analogous claim to Claim~\ref{claim:one}.

\begin{proof}[Proof of Lemma~\ref{lem:main2}]
Let $M=(\Lambda,\Sigma,\Sigma_a,m_0,m_2)$ be a deterministic tree automaton recognizing $L$. We may assume that all states of $M$ are reachable, i.e.\ for each $s\in\Sigma$ there is a tree $T$ such that $M(T)=s$. The key for this proof is to find the right notion of ``strongly connected component'' for tree automata, noting that tree automata are not naturally directed graphs. Define $\widehat{M}$ to be the partially labeled directed graph where:
\begin{itemize}
    \item The vertices are $\Sigma\times\Lambda\sqcup \Sigma^2$. Call the vertices in $\Sigma\times\Lambda$ \textit{tree-like} and those in $\Sigma^2$ \textit{forest-like}.
    \item There is a directed edge labeled $\lambda$ from forest-like $(s_1,s_2)$ to tree-like $(t,\lambda)$ when $f(\lambda,s_1,s_2)=t$, and unlabeled directed edge from tree-like $(t,\lambda)$ to forest-like $(t,s)$ and $(s,t)$. Call edges of the first type \textit{gluing edges} and edges of the second type \textit{union edges}.
\end{itemize}
Note that $\widehat{M}$ is bipartite in the sense that there are no edges between two tree-like vertices or two forest-like vertices. The vertices of $L(\widehat{M})$ are called \textit{gluing} or \textit{union} if they correspond to gluing or union edges in $\widehat{M}$, respectively. $L(\widehat{M})$ is also bipartite: there are no edges (in $L(\widehat{M})$) between two gluing vertices or two union vertices.

Say that a vertex of $L(\widehat{M})$ is \textit{post-initial} if it is a union vertex that corresponds to an out-edge from tree-like $(t,\lambda)\in V(\widehat{M})$ with $m_0(\lambda)=t$. Let $J$ be the following graph. Delete from $L(\widehat{M})$ any vertex not reachable from post-initial vertices. Then add one vertex with label $\lambda$ for each $\lambda\in\Lambda$, with an out-edge from this vertex $\ell$ to each union vertex $uv$ whose corresponding edge in $\widehat{M}$ is an in-edge of forest-like $(m_0(\lambda),t)$ or $(t,m_0(\lambda))$ for some $t\in\Sigma$ (i.e.\ $v=(m_0(\lambda),t)$ or $(t,m_0(\lambda))$). These vertices $\ell$ of $J$ are called \textit{initial}. The out-neighbors of initial vertices of $J$ are precisely the post-initial vertices, and initial vertices have no in-neighbors. It is the scc's of $J$ that will be relevant to us.

The next modification compared to the proof of Lemma~\ref{lem:strings} is that we must consider ``tree walks'' instead of conventional walks. Let $D$ be the quotient of $J$ by its scc's $C_1,\dots,C_k$, excluding loops; $D$ is a directed acyclic graph. Say that a vertex of $D$ is \textit{initial} if some vertex in the corresponding scc is initial; note initial vertices of $D$ correspond to scc's consisting of a single initial vertex of $J$. As before, a vertex of $D$ is called \textit{loopy} if its corresponding scc of $J$ contains a closed walk and \textit{non-loopy} otherwise. Call a vertex of $D$ \textit{gluing} if the corresponding scc contains a gluing vertex of $J$ and \textit{union} if it contains a union vertex of $J$. Vertices of $D$ may be both gluing and union if they contain both gluing and union vertices; in fact, all loopy vertices are both gluing and union.

Then a \textit{tree walk} on $D$ is a rooted, labeled, \textit{not necessarily binary} tree $W$ labeled with labels from $V(D)$ where:
\begin{itemize}
    \item All leaves of $W$ are initial.
    \item All other vertices of $W$ are gluing.
    \item If $v$ is a child of $u$, with labels $C_v$ and $C_u$ respectively, then either there is an edge from $C_v$ to $C_u$ in $D$ or there is non-loopy union vertex $C$ such that there is an edge from $C_v$ to $C$ and from $C$ to $C_u$.
    \item For all $v$, no two child subtrees of $v$ are (labeled-)isomorphic.
\end{itemize}

Let $\Wc$ be the set of tree walks of $D$. This is finite and computable. This is not completely trivial, but it follows by a simple inductive argument on the height (length of the longest directed path) of $D$; in particular, inductively, a given vertex only has finitely many possible subtrees, since all subtrees use only vertices strictly closer to the initial vertices in $D$. The bound on $\abs{\Wc}$ obtained by this ends up being a power tower whose height depends on the height of $D$.

We define three trees $G_Z,\widehat{W}_Z,$ and $W_Z$, for $Z$ a tree, as follows. To each $u\in V(Z)$, if $u$ has children $v_1,v_2$, let $g_u$ be the gluing vertex corresponding to the gluing edge $(\sigma_M(v_1),\sigma_M(v_2))\to (\sigma_M(v_1), \lambda_u)$, where $\lambda_u$ is the label of $u$ in $Z$. If $u$ is a leaf, then $g_u$ is instead the initial vertex of $D$ containing the initial vertex of $J$ with label $\lambda_u$. Let $G_Z$ be the labeled graph formed by replacing, for each $u\in V(Z)$, the label of $u$ with $g_u$. Let $C_u$ be the vertex of $D$ that corresponds to the scc containing $g_u$. Notice that the subgraph of $Z$ spanned by vertices $v$ with $C_v=C$ is connected (or empty) for each $C\in V(D)$. Let $\widehat{W}_Z$ be the labeled directed (not necessarily binary) tree obtained from $Z$ by contracting each such subgraph to a single vertex, labeling it with the corresponding vertex from $D$ (i.e.\ scc of $\widehat{M}$). This is not yet a tree walk since there may be a vertex with two isomorphic subtrees. Let $W_Z$ be the tree walk obtained by repeatedly deleting any duplicate subtrees in $\widehat{W}_Z$ until none remain. This deletion process results in the same tree walk regardless of the order of the deletions.

Note that it is not the case that every tree walk is the tree walk of some $Z$. This differs from the case for string languages and will require another step at the end of this proof, Claim~\ref{claim:threeprime}.

Let $\Wc_X$ be the set $\{W_Z\mid Z\in L,X\preceq Z\}$. We claim, as in the case of string languages:
\renewcommand{\theclaim}{\arabic{claim}'}
\begin{claim}
For any $X,Y$, there is a $Z\in L$ with $X,Y\preceq Z$ if and only if $\Wc_X\cap\Wc_Y\ne\varnothing$.
\end{claim}
For the forward direction, if $X,Y\preceq Z\in L$ then $W_Z\in \Wc_X\cap \Wc_Y$. For the reverse direction, suppose that $W\in \Wc_X\cap \Wc_Y$. We will find a tree in $L$ that contains both $X$ and $Y$ topologically. First, we will define a tree $T=f_{n,n}(W)$ such that $T$ contains \textit{all} trees $U$ that have at most $n$ vertices with $W\in \Wc_U$. The meaning of the second ``$n$'' in $f_{n,n}(W)$ will be explained later. It will not be the case in general that $W_T=W$; instead, only $W\preceq W_T$ holds, but this will not be an issue. Taking $n=\max\{\abs{V(X)}, \abs{V(Y)}\}$ will prove the claim aside from the fact that $T$ might not be in $L$, which requires an easy modification at the end.

For each gluing $a,b\in C$, when $C\in V(D)$ is loopy, let $T_{a,b}$ be a $\Lambda$-labeled tree with root $r$ and some vertex $v_{a,b}$ such that the labels of $r$ and $v_{a,b}$ in $G_{T_{a,b}}$ are $b$ and $a$, respectively. We call these trees \textit{transformer trees}. Additionally, for each gluing or initial vertex $a$, let $T_a$ be a tree with root $r$ so that $r$ is labeled $a$ in $G_{T_a}$. We call these trees \textit{base trees}. The tree $f_{n,n}(W)$ will ultimately be formed by ``gluing together'' a bunch of base trees, using transformer trees to modify which gluing vertices get used.

Suppose $C$ is loopy, $T$ has the root of $G_T$ labeled $a\in C$, and $b$ is a given vertex in $C$. Then the tree $T'$ formed by replacing the subtree of $T_{a,b}$ rooted at $v_{a,b}$ with $T$ is called the \textit{$b$-transformation} of $T$; it has the property that the root of $G_{T'}$ is labeled $b$.

Now let us define $f_{n,m}(W)$ inductively in $n,m$ and the depth of $W$, assuming that $W=W_Z$ for some $Z$:
\begin{itemize}
    \item Case 1: Any $n$; any $m$; $W$ has one vertex $u$. Then $u$'s label $C=\{a\}$ is initial in $D$, where $a$ is initial in $J$. Then $f_{n,m}(W)$ is defined to be the base tree $T_a$.
    \item Case 2: Any $n$; $m=0$; $W$ has root $u$ with children $v_1,\dots,v_k$. Let $C_1,\dots,C_k$ be the labels of the children of $u$ and $W_1,\dots,W_k$ the subtrees rooted at the children of $u$.
    \begin{itemize}
        \item Case 2a: $u$'s label $C=\{c\}$ is not loopy. Then by the assumption that $W=W_Z$ is a tree walk, it must be that $u$ has exactly two children (i.e.\ $k=2$), and for $i=1,2$ there is some gluing or initial $b_i\in C_i$ so that there is some union vertex $b'$ and an edge from $b_i$ to $b_i'$ and $b_i'$ to $c$ in $J$. If $C_i$ is loopy, let $T_i$ be the $b_i$-transformation of $f_{n,n}(W_i)$; otherwise, let $T_i=f_{n,n}(W_i)$. Let $c$ correspond to the gluing edge $(s_1,s_2)\to(t,\lambda)$. Then $f_{n,0}(W)$ is the tree given by a root of label $\lambda$ with child subtrees $T_1$ and $T_2$.
        \item Case 2b: $C$ is loopy. Let $c_1,\dots,c_\ell$ be the gluing vertices in $C$. For $1\le \alpha\le \ell$; $1\le i,j\le k$; and gluing vertices $b_\beta\in C_i$ and $b_\gamma\in C_j$:
        \begin{itemize}
            \item If there is a union vertex $b'$ such that there is an edge from $b_\beta$ to $b'$, from $b_\gamma$ to $b'$, and from $b'$ to $c_\alpha$ with label $\lambda$, then let $T_{\alpha,i,j,\beta,\gamma}$ be the tree formed by a vertex with label $\lambda$ and child subtrees given by the $b_\beta$-transformation of $f_{n,n}(W_i)$ and $b_\gamma$-transformation of $f_{n,n}(W_j)$.
            \item Otherwise let $T_{\alpha,i,j,\beta,\gamma}$ be the empty graph.
        \end{itemize}
        Let $U$ be the disjoint union of all of the $T_{\alpha,i,j,\beta,\gamma}$. Suppose $U$ has multiple connected components $T_1$ and $T_2$. Let $b_1,b_2,b\in C$, $t_1,t_2\in \Sigma$, and $\lambda_1,\lambda_2,\lambda\in\Lambda$ be such that for $p=1,2$, $b_p$ is a gluing edge to tree-like $(t_p,\lambda_p)$, and $b$ is a gluing edge from $(t_1,t_2)$ labeled $\lambda$. Then in $U$, replace $T_p$ with the $b_p$-transformation $T_p'$, then add a vertex $x$ with label $\lambda$ and make $T_1'$ and $T_2'$ be $x$'s child subtrees. (In other words, take appropriate transformations of $T_1$ and $T_2$ so that they may be made subtrees of a new vertex $x$ while remaining in the component $C$.) Repeat this process for as long as $U$ has multiple connected components. After all connected components of $U$ are combined in this manner, the resulting tree is $f_{n,0}(W)$.
    \end{itemize}
    \item Case 3: Any $n$; $m>0$; $W$ has root $u$ with children $v_1,\dots,v_k$. Let $C_1,\dots,C_k$ and $W_1,\dots,W_k$ be as before.
    \begin{itemize}
        \item Case 3a: $u$'s label $C$ is not loopy. Then $f_{n,m}(W)$ is just $f_{n,0}(W)$.
        \item Case 3b: $C$ is loopy. Let $c_1,\dots,c_\ell$ be the gluing vertices in $C$. For $1\le \alpha,\beta\le \ell$; $1\le i\le k$; and gluing vertex $b_\gamma\in C_i$:
        \begin{itemize}
            \item If there is a union vertex $b'$ such that there is an edge from $c_\beta$ to $b'$, from $b_\gamma$ to $b'$, and from $b'$ to $c_\alpha$ with label $\lambda$, then let $T_{\alpha,\beta,i,\gamma}$ be the tree formed by a vertex with label $\lambda$ and child subtrees given by the $c_\beta$-transformation of $f_{n,m-1}(W)$ and the $b_\gamma$-transformation of $f_{n,n}(W_i)$.
            \item Otherwise, $T_{\alpha,\beta,i,\gamma}$ is the empty graph.
        \end{itemize}
        Additionally, for $1\le \alpha,\beta,\gamma\le \ell$:
        \begin{itemize}
            \item If there is a union vertex $b'$ such that there is an edge from $c_\beta$ to $b'$, from $c_\gamma$ to $b'$, and $b'$ to $c_\alpha$ with label $\lambda$, then let $T_{\alpha,\beta,\gamma}$ be the tree formed by a vertex with label $\lambda$ and child subtrees given by the $c_\beta$-transformation of $f_{n,m-1}(W)$ and the $c_\gamma$-transformation of $f_{n,m-1}(W)$.
            \item Otherwise, $T_{\alpha,\beta,\gamma}$ is the empty graph.
        \end{itemize}
        Let $U$ be the disjoint union of all the $T_{\alpha,\beta,i,\gamma}$, all the $T_{\alpha,\beta,\gamma}$, and $f_{n,0}(W)$. Connect the trees together as was done at the end of case 2b. The resulting tree is $f_{n,m}(W)$.
    \end{itemize}
\end{itemize}
Say that $T$ \textit{$m$-contains} $U$ if it contains $U$ in a way such that at most $m$ vertices of $U$ are found in the subtree of $T$ that gets contracted to the root of $W_T$. Note that $U\preceq T$ if and only if $T$ $n$-contains $U$ where $n=\abs{V(U)}$. We will verify by induction on $n$, $m$, and $W$ that $f_{n,m}(W)$ contains all trees $U$ with at most $n$ vertices with the property that there is a $T'$ that $m$-contains $U$ with $W_{T'}=W\in \Wc_U$. The base case corresponds to case 1 above, and there are four cases for the inductive step, corresponding to cases 2a, 2b, 3a, and 3b above. To reduce monotony, we will just prove the base case and fourth inductive step case, which is the most complicated; all of the other inductive step cases follow similar logic.

Base case: If $W$ has one vertex $C$, then since $W=W_Z$, that vertex is initial. Initial vertices are never loopy, so the scc $C$ has just one vertex of $J$, say $a$. Then the only tree $T'$ with $W_{T'}=W$ is the base tree $T_a$, which actually must be just a single vertex with label equal to the label of $a$ in $J$. Therefore $f_{n,m}(W)$ has the desired property (for any $n$ and $m$).

Inductive step: Suppose that $T'$ is a tree that $m$-contains $U$ with $W_{T'}=W\in\Wc_U$, and $W$ has more than one vertex. We need to prove that $f_{n,m}(W)$ contains $U$. Let $\lambda$ be the label of the root of $U$. Let $x$ be the root of $W$; $C$ be the label of $x$; $C_1,\dots,C_k$ be the labels of the children of $x$; and $W_1,\dots,W_k$ be the subtrees rooted at the children of $x$, as before.

Cases 1-3: Omitted; they are all fairly similar to case 4 below. Case 1 is when $m=0$ and $C$ is not loopy, case 2 is when $m=0$ and $C$ is loopy, and case 3 is when $m>0$ and $C$ is not loopy.

Case 4: $m>0$ and $C$ is loopy. There are several subcases to consider:
\begin{itemize}
    \item Subcase 4a: $T'$ actually 0-contains $U$. Then by inductive hypothesis, $f_{n,0}(W)$ contains $U$, and by construction, $f_{n,m}(W)$ contains $f_{n,0}(W)$ so it also contains $U$.
    \item Subcase 4b: $T'$ $m$-contains $U$ but does not 0-contain $U$. In such an $m$-containment, let $T_1$ and $T_2$ be the subtrees of $T'$ that contain $U$'s root's children's subtrees $U_1$ and $U_2$, respectively, chosen so that the roots of $U_1$ and $U_2$ get mapped to the roots of $T_1$ and $T_2$. Note that $T_1$ and $T_2$ are not necessarily distinct; they may both be all of $T'$, for instance. There are now several subsubcases to consider:
    \begin{itemize}
        \item Subsubcase 4b(i): $T_1$ and $T_2$ are both proper subtrees of $T'$. Then let $t_1$ and $t_2$ be the roots of $T_1$ and $T_2$, respectively. By the assumption that $T'$ does not 0-contain $U$, the last common ancestor of $t_1$ and $t_2$ in $T'$ is the root $t$ of $T'$. Let $s_1$ and $s_2$ be the children of $t$ so that $s_1$ is an ancestor of (or equal to) $t_1$ and $s_2$ is an ancestor of (or equal to) $t_2$. Consider now the scc's $C_1'$ and $C_2'$ containing $g_{s_1}$ and $g_{s_2}$, respectively. These are equal to $C_i$ and $C_j$ for some $i$ and $j$. Then $T_{\alpha,i,j,\beta,\gamma}$ (part of the construction in case 2b) is nonempty and has root labeled $\lambda$ for some $\alpha,\beta,\gamma$, so $f_{n,0}(W)$ contains $U$, so $f_{n,m}(W)$ contains $U$.
        \item Subsubcase 4b(ii): $T_1=T'$ and $T_2$ is a proper subtree of $T'$. Then $T_1$ $(m-1)$-contains $U_1$. Let $s_2$ be the child of the root of $T'$ such that $t_2$ is a descendent of $s_2$, and let $C_2'$ be the scc containing $g_{s_2}$, as before. Then $C_2'=C_i$ for some $i$, so $T_{\alpha,\beta,i,\gamma}$ is nonempty and has root labeled $\lambda$ for some $\alpha,\beta,\gamma$. Then this tree, which is a subtree of $f_{n,m}(W)$, contains $U$, since $f_{n,m-1}(W)$ contains $U_1$ and $f_{n,n}(W_i)$ contains $U_2$ by inductive hypothesis.
        \item Subsubcase 4b(iii): $T_1$ is a proper subtree of $T'$ and $T_2=T'$. This is the same as the previous bullet point but with the roles of $T_1$ and $T_2$ swapped.
        \item Subsubcase 4b(iv): $T_1=T_2=T'$. Then $T_1$ $(m-1)$-contains $U_1$ and $T_2$ $(m-2)$-contains $U_2$. We have that $T_{\alpha,\beta,\gamma}$ is nonempty and has root labeled $\lambda$ for some $\alpha,\beta,\gamma$. Then this tree, which is a subtree of $f_{n,m}(W)$, contains $U$, since $f_{n,m-1}(W)$ contains both $U_1$ and $U_2$ by inductive hypothesis.
    \end{itemize}
\end{itemize}

This completes the induction, so we have that $f_{n,n}(W)$ contains all $U$ with $W\in\Wc_U$ and at most $n$ vertices, as desired. By construction, the root of $W_{f_{n,n}(W)}$ has the same label as the root of $W$. By assumption that $W\in \Wc_X$, we have $W=W_Z$ for some $Z\in L$, so this label $C$ has an accepting gluing vertex $a$. If $C$ is not loopy then $C=\{a\}$, so take $Z=f_{n,n}(W)$; otherwise, take $Z$ to be the $a$-transformation of $f_{n,n}(W)$. In either case we have $Z\in L$ and $Z$ contains both $X$ and $Y$ if $n=\max\{\abs{V(X)},\abs{V(Y)}\}$, completing the proof of the claim.
\qqed

\begin{claim}\label{claim:twoprime}
Consider $\Wc_X$ for arbitrary $X$. There is an $X^*$ with $\abs{V(X^*)}\le 2^{2^{\abs{\Wc}}}$ and $\Wc_X=\Wc_{X^*}$.
\end{claim}
Suppose that $\abs{V(X)} > 2^{2^{\abs{\Wc}}}$. We will find a smaller $X'$ with $\Wc_{X'}=\Wc_X$. First, say a tree walk $W'$ is a \textit{prefix} of a tree walk $W$ if it is a contiguous subtree of $W$ and the leaves of $W'$ are leaves of $W$. Define a deterministic tree automaton $N$ as follows:
\begin{itemize}
    \item The alphabet of $N$ is the same as $M$, i.e.\ $\Lambda$.
    \item The states of $N$ are $2^\Wc$.
    \item The function ``$m_0$'' for $N$, which we will call $n_0$ to avoid confusion with the $m_0$ function associated with $M$, is given by sending each $\lambda\in \Sigma$ to $\{C_\lambda\}$ where $C_\lambda$ indicates the tree walk of one vertex where that vertex is labeled according to the scc of $J$ containing the initial vertex associated to $\lambda$.
    \item The function ``$m_2$'' for $N$, which we will call $n_2$ instead, is given by
    \begin{align*}
        n_2(\lambda, S_1, S_2) = \{W\mid{}&\text{there are $W_1\in S_1,W_2\in S_2$ such that $W$ has prefixes $W_1,W_2$ and} \\
        &\text{$W$ has root labeled $C_i$ and some gluing vertex $b\in C_i$ has label $\lambda$}\}.
    \end{align*}
\end{itemize}
The accepting states of $N$ are irrelevant for our purposes. Then it is the case that $W\in \Wc_X$ if and only if all of the following are true:
\begin{itemize}
    \item a prefix of $W$ is in $N(X)$,
    \item the scc corresponding to the root of $W$ is accepting,
    \item there is some $Z$ with $W=W_Z$.
\end{itemize}
This may be proven easily by induction on $\abs{V(X)}$; we omit the details. Note that the third condition did not appear in the proof for string languages because all walks in that proof were realized by some string $z$.

The tree $X$ is so large that it has depth greater than $2^{\abs{\Wc}}$, so it has vertices $u_1,u_2$ with $u_2$ a descendent of $u_1$ and $\sigma_N(u_1)=\sigma_N(u_2)$. Then take $X'$ to be the (smaller) tree obtained from $X$ by replacing the subtree at $u_1$ with the subtree at $u_2$. We have $N(X)=N(X')$ so $\Wc_X=\Wc_{X'}$.
\qqed

\begin{claim}\label{claim:threeprime}
For any $Z$, there is a $Z^*$ with $\abs{V(Z^*)}\le 2^{\abs{\Wc}\cdot \abs{\Sigma}\cdot \abs{\Lambda}}$ and $W_Z=W_{Z^*}$.
\end{claim}
Suppose that $\abs{V(Z)}>2^{\abs{\Wc}\cdot \abs{\Sigma}\cdot \abs{\Lambda}}$. Then $Z$ has depth greater than $\abs{\Wc}\cdot \abs{\Sigma}\cdot \abs{\Lambda}$, so it has two vertices $u$ and $v$ so that:
\begin{itemize}
    \item $v$ is a descendent of $u$,
    \item $\sigma_M(v)=\sigma_M(u)$,
    \item $u$ and $v$ have the same label, and
    \item $W_{Z_u}=W_{Z_v}$ where $Z_x$ is the subtree rooted at $x$.
\end{itemize}
Then replacing the subtree rooted at $u$ with the subtree rooted at $v$ gives a smaller tree $Z'$ with $W_{Z'}=W_Z$.
\qqed

Claim~\ref{claim:twoprime} implies that in any minimal $L$-semibad pair $(X,Y)$, both $X$ and $Y$ have at most $2^{2^{\abs{\Wc}}}$ vertices (in fact, we have proved that they have depth at most $2^{\abs{\Wc}}$). The proof of Claim~\ref{claim:twoprime}, along with Claim~\ref{claim:threeprime}, also gives us a way to compute $\Wc_X$ for any tree $X$: $\Wc_X$ is the set of tree walks $W$ with a prefix in $N(X)$ and accepting root label such that $W=W_Z$ for some $Z$. Then we can compute the set of minimal $L$-semibad pairs $R_L$, which allows us to construct an automaton to recognize $B_L'$ since $B_L'$ is equal to
\[ (\Tc_k\hash\Tc_k)\cap\bigcap_{X,Y\in R_L}\overline{L_X\hash L_Y} \]
where $L_T$ is the set of strings $Z$ with $T\preceq Z$, which is explicitly regular.
\end{proof}

The upper bounds obtained from this proof are much larger than the bounds we obtained for string languages. Indeed, $\abs{\Wc}$ is bounded by a power tower whose height is a function of the number of states in $M$, so the size of minimal (semi)bad pairs and the size of the smallest automaton recognizing $B_L$ is only upper bounded by a power tower-type expression.

As an immediate consequence of this result:

\begin{corollary}
Fix an explicitly regular tree language $L$. There is a linear time algorithm to determine if a given pair $(X,Y)$ is $L$-(semi)bad or not.
\end{corollary}
\begin{proof}
Construct the tree automaton $M'$ recognizing $B_L$ (for bad pairs) or $B_L'$ (for semibad pairs). Then run $M'$ on input $X\hash Y$, which takes $O(\abs{V(X)}+\abs{V(Y)})$ time.
\end{proof}
Of course, the hidden constant of this algorithm is phenomenally large, so this is not actually a practical algorithm for recognizing (semi)bad pairs.

\section{Applications}
\label{sec:cor}

\subsection{General trees}

It is not too difficult to extend the proof of Theorem~\ref{thm:trees} (or Lemma~\ref{lem:main2}) to work for sets of trees with any bounded arity, not just binary trees. However, this argument does not immediately extend to general trees, with unbounded arity, since the definition of (finite) tree automata does not immediately extend to this case. We can still recover this result, which was Corollary~\ref{cor:generaltrees} in the introduction, with a simple trick:

\begin{proof}[Proof of Corollary~\ref{cor:generaltrees}]
To each general tree with $k$ possible labels, associate a binary tree with $k+1$ labels as follows. Call the extra label $\uparrow$. Replace each vertex $u$ with $x$ children $v_1,\dots,v_x$ with a chain $u_1,\dots,u_{x-1}$ so that the children of $u_i$ are $u_{i+1}$ and $v_i$ for $i<x-1$ and $v_{x-1}$ and $v_x$ for $i=x-1$. The label of $u_1$ is the label of $u$ and the labels for all of the other $u_i$ are $\uparrow$. Note that this binary tree is not unique; it depends on the order chosen for the children of each vertex. Nevertheless, it is the case that the language of binary trees corresponding to $S$-free general trees is explicitly regular for any finite $S$, though we omit a detailed construction. The result then follows from Lemma~\ref{lem:main}.
\end{proof}

Various extensions to this corollary are possible by modifying the tree language constructed in the proof. Of particular interest, we can consider general trees whose leaf labels are restricted to a subset of the possible labels; this will be used in the proof of Corollary~\ref{cor:p4free} below.

\subsection{Cographs}

A classic result is that any $P_4$-free graph (aka \textit{cograph}) aside from the empty graph is uniquely associated to a \textit{cotree}, which is a general tree with three kinds of labels called $\sqcup$, $\vee$, and $K_1$. Leaves are labeled $K_1$, non-leaves are labeled either $\sqcup$ or $\vee$, and no two adjacent vertices have the same label. Associate to each vertex of the tree a cograph according to the vertex label and the cographs associated to the children: leaves $K_1$ are associated to the graph $K_1$, vertices labeled $\sqcup$ are associated to the disjoint union of the graphs associated to their child subtrees, and vertices labeled $\vee$ are associated to the ``join'' of their child subtrees. The \textit{join} of graphs $G_1,\dots,G_n$ is the graph obtained from $\bigsqcup G_i$ by adding every possible edge between $G_i$ and $G_j$ for $i\ne j$.

It is the case that if $G_1$ and $G_2$ are cographs and $T_1$ and $T_2$ are the cotrees associated to $G_1$ and $G_2$, respectively, then $G_1\subseteq G_2$ if and only if $T_1\preceq T_2$ \cite{cotree}. Now a proof of Corollary~\ref{cor:p4free} may be obtained by the same idea as the proof of Corollary~\ref{cor:generaltrees} but modifying the tree language so that the only allowed leaf label is $K_1$, no non-leaf is labeled $K_1$, and no two adjacent vertices have the same label (all of these modifications still result in an explicitly regular tree language).

\subsection{Bounded treewidth families}

Let $(T,X)$ be a tree decomposition of $G$, i.e.
\begin{itemize}
    \item $T$ is a tree and $X:V(T)\to 2^{V(G)}$ is a function.
    \item For all $v\in V(G)$, there is a $t\in V(T)$ with $v\in X(t)$.
    \item For all $uv\in E(G)$, there is a $t\in V(T)$ with $u,v\in X(t)$.
    \item For all $v\in V(G)$, $X^{-1}(v)$ is a connected subgraph of $T$.
\end{itemize}
The sets $X(t)$ for $t\in V(T)$ are called \textit{bags} of the tree decomposition. We may further assume, by duplicating certain vertices of $T$, that $T$ is a binary tree. We say that a family $\Fc$ has \textit{bounded treewidth} if every graph $G\in\Fc$ admits a tree decomposition $(T,X)$ with $\max\{\abs{X(t)}\}\le c$, $c$ a constant. In this case, tree decompositions may be viewed as vertex- and edge-labeled rooted binary trees, with the possible vertex labels being the graphs on $\le c$ vertices (\textit{not} up to isomorphism, i.e.\ these graphs are labeled with labels $1,\dots,c$) and the possible edge labels being the set of bipartite graphs on $\le c+c$ vertices (again, not up to isomorphism; the edge labels are used to match vertices in adjacent bags that correspond to the same vertex in $G$). By subdividing edges and adding dummy leaves, and choosing an arbitrary non-leaf as the root, we may then associate to a tree decomposition a vertex-labeled rooted binary tree, i.e.\ what we have been calling a ``tree'' in this paper.

It is the case that the set of trees associated to (bounded width, binary) tree decompositions of $S$-minor-free graphs is explicitly regular, though we omit a detailed construction. This fact may also be seen by appealing to the tree automaton construction of Courcelle \cite{courcelle}, since containing a given graph as a minor is an $\text{MSO}_2$ property. Here we must use the assumption that $\Forb_\preceq(S)$ has bounded treewidth so that the $S$-minor-free graphs uniformly admit such tree decompositions. Then Corollary~\ref{cor:tw} follows by using Lemma~\ref{lem:main} to determine if there are any tree decompositions that correspond to $\Forb_{\preceq}(S)$-bad pairs.

\subsection{Bounded cliquewidth families}

Recall that the \textit{cliquewidth} of a graph $G$ is the fewest number of labels needed to construct $G$ by the following operations:
\begin{enumerate}
    \item Introduction of a vertex with a given label $i$.
    \item Disjoint union.
    \item Adding an edge between every vertex of label $i$ and label $j$, for chosen $i\ne j$.
    \item Changing all labels $i$ to $j$.
\end{enumerate}
Similarly to cotrees and tree decompositions, graphs of bounded cliquewidth admit a description by a \textit{parse tree} structure. Here the leaves have $k$ possible labels and correspond to the first kind of operation, and nonleaves have the following possible labels:
\begin{itemize}
    \item $\sqcup$, corresponding to operation (2) above;
    \item $\vee_{i,j}$, corresponding to operation (3) above; and
    \item $i\to j$, corresponding to operation (4) above.
\end{itemize}

It is the case that the set of parse trees of $S$-free graphs with cliquewidth at most $k$ is an explicitly regular language for any finite set of graphs $S$, though once again we omit a detailed construction. As in the case of bounded treewidth families, this may immediately be seen by appealing to Courcelle's theorem for bounded cliquewidth, using the fact that containing a given graph as an induced subgraph is $\text{MSO}_1$ \cite{courcelle2}. Corollary~\ref{cor:cw} then follows from Lemma~\ref{lem:main}.

Note that bounded treewidth families also have bounded cliquewidth, so this implies that it is also decidable to determine if a bounded treewidth family has JEP under the induced subgraph relation.

\section{Further discussion}
\label{sec:more}

If $(\Fc,\le)$ is wqo, then by a similar argument to the one used in proof of Proposition~\ref{prop:implicit}, there are only ever finitely many minimal (under $\le$) $\Fc$-bad pairs. One might hope to use more precise information about the ordering to get a computable upper bound on the size any such bad pair. Of course, $\Fc$ might not have any reasonable natural notion of ``size''. For instance, infinite trees are still wqo under topological containment \cite{infinitetree}. In this case, however, Turing machines cannot even read the relevant input in finite time, so it is clear that this case must be excluded somehow if we wish to generalize the results in this paper to other wqo families.

If $(\Fc,\le)$ is wqo and there was a suitable way to define automata on members of $\Fc$, then the proof idea of Proposition~\ref{prop:implicit} would likely generalize to show that the set of $L$-bad pairs is (implicitly) ``regular'' for any ``regular $\Fc$-language'' $L$. To be more precise, the properties required for a notion of regularity to work in this sense are:
\begin{enumerate}
    \item The set $\Forb_\le(S)$ is regular for all $S$.
    \item Regularity is closed under union, intersection, and complement.
    \item It is decidable to determine if a given regular language is empty or not.
    \item $(\Fc^2, \le_2)$ also has a similar generalization of ``regular'', where $(X_1,Y_1)\le_2(X_2,Y_2)$ if $X_1\le X_2$ and $Y_1\le Y_2$.
\end{enumerate}
In property (3), the meaning of word ``given'' is related to the difference between implicitly and explicitly regular language families; in the case of regular string languages for example, the information that should be provided is a finite automaton recognizing the language in question (though one could also provide a regular expression or set of production rules and the problem remains decidable). Property (4) is really just stating that there should be an analogue of the $\#$ operation. Tree automata have all of these properties in the case $(\Fc,\le)=(\Tc_k,\preceq)$. On the other hand, the author is not aware of any such notion of automaton when $(\Fc,\le)=(\Gc,\preceq)$, the set of graphs ordered by the minor relation.

Another possible idea for proving Lemma~\ref{lem:main2}(or~\ref{lem:main}), which we did not pursue, is to ``inline'' a proof of Kruskal's theorem in order to find an explicit (computable) bound on the size of a minimal bad pair. If this idea works, then one could hope more generally to turn a proof that $(\Fc,\le)$ is wqo into a proof that it is decidable to determine if $\Forb_\le(S)$ has JEP. One reason to be skeptical that this approach wouldwork is that the bounds we obtain in the proof of Lemma~\ref{lem:main2} seem to be of a different shape than sorts of functions obtained by exploiting the fact that trees are wqo under topological containment, e.g.\ Friedman's tree and TREE functions.

Due to these challenges, we have extremely low confidence in Conjecture~\ref{conj:wqo} as stated. We are really asking if there are mild assumptions on a wqo set that make JEP decidable.

On the other hand, we suspect Corollary~\ref{cor:p4free} is sharp in the following sense:
\begin{conjecture}\label{conj:oneforbidden}
Let $H$ be a graph. Then Problem~\ref{prob:jep} is decidable for all finite $S$ with $H\in S$ if and only if $H\subseteq P_4$.
\end{conjecture}
This paper proved the ``if'' direction. Note as well that $H\subseteq P_4$ if and only if $\Forb_\subseteq(H)$ is wqo under $\subseteq$. We also conjecture the following, which is weaker than Conjecture~\ref{conj:wqo}:
\begin{conjecture}\label{conj:manyforbidden}
Let $S'$ be a finite set of graphs. Then Problem~\ref{prob:jep} is decidable for all finite $S$ with $S'\subseteq S$ if $\Forb_\subseteq(S')$ is wqo by $\subseteq$.
\end{conjecture}
Note that the converse to this statement is false, as we now describe. Recall that it is decidable to determine if bounded treewidth families have JEP under the induced subgraph relation. Let $S'$ be the set of graphs on four vertices where at least one vertex has degree three. Then $\Forb_\subseteq(S')$ is precisely the set of graphs with maximum degree at most 2, i.e.\ disjoint unions of paths and cycles. This class has bounded treewidth, so it is decidable to determine if $\Forb_\subseteq(S)$ has JEP for any finite $S$ containing $S'$.\footnote{It is also easy to see this without considering treewidth, since in all cases the $S$-free graphs admit a fairly simple description.} On the other hand, $\Forb_\subseteq(S')$ is not wqo by $\subseteq$ because cycles form an infinite antichain.

In \cite{conj}, it was conjectured that if $\Forb_\subseteq(S')$ is wqo by $\subseteq$, where $S'$ is finite, then this family also has bounded cliquewidth. This conjecture implies Conjecture~\ref{conj:manyforbidden} as a consequence of Corollary~\ref{cor:cw}.

\section{Acknowledgements}

We thank Nicolas Trotignon for bringing this problem to our attention.

\printbibliography

\end{document}